\documentclass[final,3p,times]{elsarticle}

\usepackage{amsmath,amsthm,amscd,amsfonts,amssymb,epic,eepic,bbm,graphicx,ytableau}
\usepackage{tikz}
\usepackage{tkz-euclide}
\usepackage{rotating}
\usetikzlibrary{angles,quotes}

\newtheorem{thm}{Theorem}[section]
\newtheorem{cor}[thm]{Corollary}

\newtheorem{lem}[thm]{Lemma}
\newtheorem{defn}[thm]{Definition}

\newtheorem{prop}[thm]{Proposition}

\newtheorem{rem}[thm]{Remark}
\newtheorem{example}[thm]{Example}
\theoremstyle{remark}
\newcommand{\la}{\lambda}
\def\multiset#1#2{\ensuremath{\left(\kern-.5em\left(\genfrac{}{}{0pt}{}{#1}{#2}\right)\kern-.5em\right)}}

\journal{}

\begin{document}

\begin{frontmatter}



\title{The $(s,s+d,\dots,s+pd)$-core partitions and the rational Motzkin paths}


\author[1]{Hyunsoo Cho}
\ead{coconut@yonsei.ac.kr}
\author[2]{JiSun Huh\corref{cor1}}
\ead{hyunyjia@g.skku.edu}
\author[1]{Jaebum Sohn}
\ead{jsohn@yonsei.ac.kr}
\cortext[cor1]{Corresponding author.}

\address[1]{Department of Mathematics, Yonsei University, Seoul 03722, Republic of Korea}
\address[2]{Applied Algebra and Optimization Research Center, Sungkyunkwan University, Suwon 16420, Republic of Korea}

\begin{abstract}
In this paper, we propose an $(s+d,d)$-abacus for $(s,s+d,\dots,s+pd)$-core partitions and establish a bijection between the $(s,s+d,\dots,s+pd)$-core partitions and the rational Motzkin paths of type $(s+d,-d)$. This result not only gives a lattice path interpretation of the $(s,s+d,\dots,s+pd)$-core partitions but also counts them with a closed formula. Also we enumerate $(s,s+1,\dots,s+p)$-core partitions with $k$ corners and self-conjugate $(s,s+1,\dots,s+p)$-core partitions.   
\end{abstract}

\begin{keyword}
simultaneous core partitions \sep self-conjugate \sep corners  \sep rational Motzkin paths \sep generalized Dyck paths



\MSC[2008] 05A17 \sep 05A19

\end{keyword}

\end{frontmatter}



\section{Introduction}

A partition $\la=(\la_1,\la_2,\dots,\la_\ell)$ of a positive integer $n$ is a finite non-increasing sequence of positive integer parts $\la_i$ such that $\la_1+\la_2+\cdots+\la_\ell=n$. The Young diagram of $\la$ is a finite collection of $n$ boxes arranged in left-justified rows, with the $i$th row having $\la_i$ boxes.
For the Young diagram of $\la$, the partition $\la'=(\la'_1,\la'_2,\dots)$ is called the conjugate of $\la$, where $\la'_j$ denotes the number of boxes in the $j$th column.
For each box of the Young diagram in coordinates $(i,j)$ (that is, the box in the $i$th row and $j$th column), the hook length $h(i,j)$ is the number of boxes weakly below and strictly to the right of the box. For a partition $\la$, the {\it beta-set} of $\la$, denoted $\beta(\la)$, is defined as the set of first column hook lengths of $\la$. For example, the conjugate of $\la=(5,4,2,1)$ is $\la'=(4,3,2,2,1)$ and the {\it beta-set} of $\la$ is $\beta(\la)=\{8,6,3,1\}$. See Figure \ref{fig:abacus} for the Young diagram of $\la$ and the hook lengths of each box.

For a positive integer $t$, a partition $\la$ is a $t$-core (partition) if it has no box of hook length $t$. In the previous example, $\la$ is a $t$-core for $t=5,7,$ or $t\geq 9$. For distinct positive integers $t_1,t_2,\dots,t_p$, we say that a partition $\la$ is a $(t_1,t_2,\dots,t_p)$-core if it is simultaneously a $t_1$-core, a $t_2$-core, \dots, and a $t_p$-core. The study of core partitions arose from the representation theory of the symmetric group $S_n$ (see \cite{JK} for details). Researches on simultaneous core partitions are motivated by the following result of Anderson \cite{Anderson}.

\begin{thm}\cite[Theorem 1]{Anderson}
For relatively prime positive integers $s$ and $t$, the number of $(s,t)$-core partitions is 
\[
\frac{1}{s+t}\binom{s+t}{s}\,.
\]
In particular, the number of $(s,s+1)$-core partitions is the $s$th Catalan number $C_s=\frac{1}{s+1}\binom{2s}{s}=\frac{1}{2s+1}\binom{2s+1}{s}$.
\end{thm}

Since the work of Anderson, results on $(s,t)$-cores were published by many researchers (see \cite{AHJ,CHW,Fayers,Fayers2,FV,FMS}). Also, some researchers concerned with simultaneous core partitions whose cores line up in arithmetic progression (see \cite{AL,BNY,CHS,Wang,YYZ,YZZ}).
Yang-Zhong-Zhou \cite{YZZ} showed that the number of $(s,s+1,s+2)$-core partitions is the $s$th Motzkin number $M_s=\sum_{k=0}^{\lfloor s/2 \rfloor} \binom{s}{2k}C_k$, where $C_k$ is the $k$th Catalan number. 
Amdeberhan-Leven \cite{AL} and Wang \cite{Wang} extended this result.

\begin{thm}\cite[Theorem 4.2]{AL}\label{thm:gendyck}
An \emph{$(s,p)$-generalized Dyck path} is a path from $(0,0)$ to $(s,s)$ stays above the line $y=x$ and consists of vertical steps $(0,p)$, horizontal steps $(p,0)$, and diagonal steps $(i,i)$ for $1\leq i\leq p-1$.

For integers $s$ and $p$, the number $C_s^{(p)}$ of $(s,s+1,\dots,s+p)$-core partitions is equal to the number of $(s,p)$-generalized Dyck paths, which satisfies the following recurrence relation:
\[
C_s^{(p)}=\sum_{k=1}^{s}C_{k-p}^{(p)}C_{s-k}^{(p)}\,,
\]
where $C_s^{(p)}=1$ for $s<0$.
\end{thm}

\begin{thm}\cite[Theorem 1.6]{Wang}\label{thm:Wang}
For relatively prime positive integers $s$ and $d$, the number of $(s,s+d,s+2d)$-core partitions is 
\[
\frac{1}{s+d}\sum_{k=0}^{\lfloor s/2 \rfloor}\binom{s+d}{k,k+d,s-2k}\,.
\]
\end{thm}

Recently, Baek-Nam-Yu \cite{BNY} obtained an alternative proof for Theorem \ref{thm:Wang} and found a formula for the number of $(s,s+d,s+2d,s+3d)$-core partitions.

\begin{thm}\cite[Theorem 5.7]{BNY}\label{thm:BNY}
For relatively prime positive integers $s$ and $d$, the number of $(s,s+d,s+2d,s+3d)$-core partitions is 
\[
\frac{1}{s+d}\sum_{k=0}^{\lfloor s/2 \rfloor}\left\{\binom{s+d-k}{k}+\binom{s+d-k-1}{k-1}\right\}\binom{s+d-k}{s-2k}\,.
\]
\end{thm}

According to previous literatures, there are two useful methods to count simultenous core partitions. An idea of counting paths was used in \cite{AL,Anderson,AHJ,CHW,FMS,HW} and counting the lattice points method was used in \cite{BNY,Fayers,FV,Johnson,Wang}. 

In this paper, we define the ``rational Motzkin path", which generalizes the idea of the Motzkin path, and give a generalization of Theorems \ref{thm:gendyck}, \ref{thm:Wang}, and \ref{thm:BNY} by using rational Motzkin paths of type $(s+d,-d)$ with a specific restriction (see Definition \ref{def:rational}). The following is the main result of this paper.    
 
\begin{thm}\label{thm:main}
Let $s$ and $d$ be relatively prime positive integers. For a given integer $p\geq2$, the number of $(s,s+d,\dots,s+pd)$-core partitions is equal to the number of rational Motzkin paths of type $(s+d,-d)$ 
without $UF^iU$ steps for $i=0,1,\dots,p-3$ if $p\geq 3$, that is 
\[
\frac{1}{s+d}\binom{s+d}{d}+\sum_{k=1}^{\lfloor s/2 \rfloor}\sum_{\ell=0}^{r}\frac{1}{k+d}\binom{k+d}{k-\ell}\binom{k-1}{\ell}\binom{s+d-\ell(p-2)-1}{2k+d-1}\,,
\]
where $r=\min(k-1,\,\lfloor (s-2k)/(p-2) \rfloor)$.
\end{thm}

For example, if we put $s=3$, $d=2$, and $p=2$ in Theorem \ref{thm:main}, then we get six, the number of $(3,5,7)$-core partitions. In fact, $(3,5,7)$-core partitions are \,$\emptyset, ~(1), ~(1,1), ~(2), ~(2,1,1)$, and $(3,1)$.

As a corollary, by setting $d=1$, we obtain a closed formula for the number of $(s,s+1,\dots,s+p)$-cores. Also, we give a bijection between the set of $(s,p)$-generalized Dyck paths and that of Motzkin paths of length $s$ with a restriction. Furthermore, we count the number of $(s,s+1,\dots,s+p)$-cores with $k$ corners. At the end of this paper, we enumerate self-conjugate $(s,s+1,\dots,s+p)$-core partitions.


\section{Counting $(s,s+d,\dots,s+pd)$-core partitions}\label{sec:2}

In this section, we introduce the $(s+d,d)$-abacus diagram and the rational Motzkin paths, and by using these objects we shall count $(s,s+d,\dots,s+pd)$-core partitions for relatively prime positive integers $s$ and $d$. In the sequel, we assume that $p$ is a positive integer.


\subsection{The $(s+d,d)$-abacus diagram}  

James-Kerber \cite{JK} introduced the abacus diagram which has played important roles in the theory of core partitions (see \cite{AHJ,Fayers,Fayers2,Johnson2,NS}).
The \emph{$s$-abacus diagram} is a diagram with infinitely many rows labeled $0,1,\dots,\infty$ and $s$ columns labeled $0,1,\dots,s-1$, whose position in $(i,j)$ is labeled by $si+j$, where $i=0,1,\dots,\infty$ and $j=0,1,\dots,s-1$. The \emph{$s$-abacus} of a partition $\la$ is obtained from the $s$-abacus diagram by placing a bead on each position which the number at this position belongs to $\beta(\la)$. Positions without beads are called spacers.

\begin{example}
If $\la=(5,4,2,1)$, then $\la$ is a $5$-core partition and $\beta(\la)=\{8,6,3,1\}$. Figure \ref{fig:abacus} shows the Young diagram with the hook lengths and the $5$-abacus of $\la$.
\end{example}

\begin{figure}[ht!]
\centering
\begin{tikzpicture}[scale=.44]
\node at (-7,2) {
\begin{ytableau}
8&6&4&3&1 \\
6&4&2&1 \\
3&1 \\
1
\end{ytableau}};

\foreach \i in {0,1,2,3,4}
\node at (2*\i,0) {$\i$};
\foreach \i in {5,6,7,8,9}
\node at (2*\i-10,1) {$\i$};
\foreach \i in {10,11,12,13,14}
\node at (2*\i-20,2) {$\i$};
\foreach \i in {15,16,17,18,19}
\node at (2*\i-30,3) {$\i$};
\node at (4,4.1) {\vdots};
\draw (2,0) circle (12pt);
\draw (6,0) circle (12pt);
\draw (2,1) circle (12pt);
\draw (6,1) circle (12pt);
\end{tikzpicture}

\caption{The Young diagram of a $5$-core partition $(5,4,2,1)$ with the hook lengths and its $5$-abacus}\label{fig:abacus}
\end{figure}

It is well-known that $\la$ is an $s$-core if and only if the $s$-abacus of $\la$ has no spacer below a bead in any column. Equivalently, one can have the following.

\begin{lem}\label{lem:core}\cite[Lemma 2.7.13]{JK} For a partition $\la$, $\la$ is an $s$-core if and only if $x\in\beta(\la)$ implies $x-s\in\beta(\la)$.
\end{lem}

We now introduce the $(s+d,d)$-abacus diagram, which generalizes the definition of the $s$-abacus diagram and the two-way abacus diagram suggested by Anderson \cite{Anderson}.

\begin{defn} 
Let $s$ and $d$ be relatively prime positive integers. The \emph{$(s+d,d)$-abacus diagram} is a diagram with infinitely many rows labeled $-\infty,\dots,-1,0,1,\dots,\infty$ and $s+d+1$ columns labeled $0,1,\dots,s+d$ whose position in $(i,j)$ is labeled by $(s+d)i+dj$, where $i=-\infty,\dots,\infty$ and $j=0,\dots,s+d$. For a partition $\la$, the \emph{$(s+d,d)$-abacus} of $\la$ is obtained from the $(s+d,d)$-abacus diagram by placing a \emph{bead} on each position which the number at this position  belongs to $\beta(\la)$. Again, a position without bead is called a \emph{spacer}.
\end{defn}

\begin{example}\label{ex:abacus} If $\la=(6,4,3,1,1,1,1)$, then $\la$ is a $(5,8,11)$-core partition and $\beta(\la)=\{12,9,7,4,3,2,1\}$. Figure \ref{fig:abacus2} shows the Young diagram with the hook lengths and the $(8,3)$-abacus of $\la$.

\begin{figure}[ht!]
\centering
\begin{tikzpicture}[scale=.44]
\node at (-8,0) {
\begin{ytableau}
12&7&6&4&2&1 \\
9&4&3&1 \\
7&2&1 \\
4\\
3\\
2\\
1
\end{ytableau}};

\foreach \i in {0,3,6,9,12,15,18,21,24}
\node at (\i/1.5,0) {$\i$};
\foreach \i in {8,11,14,17,20,23,26,29,32}
\node at (\i/1.5-8/1.5,1) {$\i$};
\foreach \i in {16,19,22,25,28,31,34,37,40}
\node at (\i/1.5-16/1.5,2) {$\i$};
\foreach \i in {24,27,30,33,36,39,42,45,48}
\node at (\i/1.5-24/1.5,3) {$\i$};
\foreach \i in {-8,-5,-2,1,4,7,10,13,16}
\node at (\i/1.5+8/1.5,-1) {$\i$};
\foreach \i in {-16,-13,-10,-7,-4,-1,2,5,8}
\node at (\i/1.5+16/1.5,-2) {$\i$};
\foreach \i in {-24,-21,-18,-15,-12,-9,-6,-3,0}
\node at (\i/1.5+24/1.5,-3) {$\i$};
\node at (8,4.1) {\vdots};
\node at (8,-3.9) {\vdots};
\draw (8,0) circle (12pt);
\draw (6,0) circle (12pt);
\draw (10,-1) circle (12pt);
\draw (8,-1) circle (12pt);
\draw (6,-1) circle (12pt);
\draw (2,0) circle (12pt);
\draw (12,-2) circle (12pt);
\end{tikzpicture}
\caption{The Young diagram of the partition $(6,4,3,1,1,1,1)$ with the hook lengths and its $(8,3)$-abacus}\label{fig:abacus2}
\end{figure}
\end{example}

This modified abacus diagram is useful when we consider $(s,s+d,\dots,s+pd)$-core partitions with $p\geq 2$. For a given $(s,s+d,\dots,s+pd)$-core partition $\la$, if we consider the $(s+d,d)$-abacus of $\la$, then the position in $(i,j)$ is a bead implies that positions in $(i-1,j-p+1),(i-1,j-p+2),\dots,(i-1,j+1)$ are also beads if these positions are labeled by positive numbers as in Figure \ref{fig:abacus2}. We now have the following lemma.

\begin{lem}\label{lem:af}
For a given $p\geq2$ and the $(s+d,d)$-abacus of an $(s,s+d,\dots,s+pd)$-core partition $\la$, we define a function 
\[
f:\{0,1,\dots,s+d\}\rightarrow \{-\infty,\dots,-1,0,1,\dots,\infty\}
\]
as follows:
For a column number $j$, $f(j)$ is defined to be the row number $i$ such that the position in $(i,j)$ is a spacer which is labeled by the smallest nonnegative number in column $j$. Then $f$ satisfies the followings: 

\begin{enumerate}
\item[(a)] $f(0)=0$ and $f(s+d)=-d$.
\item[(b)] $f(j-1)$ is exactly one of the values $f(j)-1$, $f(j)$, and $f(j)+1$, for $1\leq j \leq s+d$.
\item[(c)] If $f(j-1)=f(j)-1$, then $f(j-p+1),f(j-p+2),\dots,f(j-2)\geq f(j-1)$, for $p-1 \leq j\leq s+d$.

\end{enumerate} 
\end{lem}

\begin{proof}
$(a)$ is clear by the definition. Note that $f(j)=i$ implies that the position in $(i-1,j)$ is either a bead or labeled by a negative number. If the position in $(i-1,j)$ is a bead, then the position in $(i-2,j-1)$ is also either a bead or labeled by a negative number by Lemma \ref{lem:core}. Otherwise, the position in $(i-1,j-1)$ is also labeled by a negative number and $f(j-1)=i$. Therefore, $f(j-1)\geq f(j)-1$. We can also see that $f(j)\geq f(j-1)-1$. By combining these two inequalities, we conclude the result in $(b)$. Similarly, $f(j-p+1),f(j-p+2),\dots,f(j-2)\geq f(j)-1$. It follows that if $f(j-1)=f(j)-1$, then $f(j-p+1),f(j-p+2),\dots,f(j-2) \geq f(j-1)$.
\end{proof}

In Figure \ref{fig:abacus2}, since the smallest nonnegative numbers labeled by a spacer in each column are $0,11,6,17,20$,
$15,10,5,0$, respectively, so $f(0)=0$, $f(1)=1$, $f(2)=0$, $f(3)=1$, $f(4)=1$, $f(5)=0$, $f(6)=-1$, $f(7)=-2$, and $f(8)=-3$. We see that $f$ agrees with $(a)$ and $(b)$ of Lemma \ref{lem:af}.

\subsection{Rational Motzkin paths of type $(s,t)$}

A \emph{Motzkin path of length $s$} is a path from $(0,0)$ to $(s,0)$ stays above the $x$-axis and consists of up steps $(1,1)$, down steps $(1,-1)$, and flat steps $(1,0)$. We introduce a path which generalizes the idea of the Motzkin path. 

\begin{defn} \label{def:rational}
A \emph{rational Motzkin path of type $(s,t)$} is a path from $(0,0)$ to $(s,t)$ stays above the line $y=tx/s$ and consists of up steps $U=(1,1)$, down steps $D=(1,-1)$, and flat steps $F=(1,0)$. If a rational Motzkin path of type $(s,t)$ is allowed to go below the line $y=tx/s$, then it is called a \emph{free rational Motzkin path} of type $(s,t)$. 
\end{defn}

\begin{figure}[ht!]
\centering
\begin{tikzpicture}[scale=.34]
\foreach \i in {0,1,2,3,4,5}
\draw[dotted, gray!60] (\i,-2)--(\i,1);

\foreach \i in {0,1,2,3}
\draw[dotted, gray!60] (0,-2+\i)--(5,-2+\i);

\draw[->, black!70] (0,0)--(6,0);
\draw[->, black!70] (0,-3)--(0,2);
\draw[dotted, gray!60] (0,0)--(5,-2);

\node[below, black!70] at (6,0) {$x$};
\node[left,black!70] at (0,2) {$y$};

\draw[thick] (0,0)--(1,0)--(2,0)--(3,-1)--(4,-1)--(5,-2);
\end{tikzpicture}
\begin{tikzpicture}[scale=.34]
\foreach \i in {0,1,2,3,4,5}
\draw[dotted, gray!60] (\i,-2)--(\i,1);

\foreach \i in {0,1,2,3}
\draw[dotted, gray!60] (0,-2+\i)--(5,-2+\i);

\draw[->, black!70] (0,0)--(6,0);
\draw[->, black!70] (0,-3)--(0,2);
\draw[dotted, gray!60] (0,0)--(5,-2);

\node[below, black!70] at (6,0) {$x$};
\node[left,black!70] at (0,2) {$y$};

\draw[thick] (0,0)--(1,0)--(2,0)--(3,0)--(4,-1)--(5,-2);
\end{tikzpicture}
\begin{tikzpicture}[scale=.34]
\foreach \i in {0,1,2,3,4,5}
\draw[dotted, gray!60] (\i,-2)--(\i,1);

\foreach \i in {0,1,2,3}
\draw[dotted, gray!60] (0,-2+\i)--(5,-2+\i);

\draw[->, black!70] (0,0)--(6,0);
\draw[->, black!70] (0,-3)--(0,2);
\draw[dotted, gray!60] (0,0)--(5,-2);

\node[below, black!70] at (6,0) {$x$};
\node[left,black!70] at (0,2) {$y$};

\draw[thick] (0,0)--(1,1)--(2,0)--(3,0)--(4,-1)--(5,-2);
\end{tikzpicture}
\begin{tikzpicture}[scale=.34]
\foreach \i in {0,1,2,3,4,5}
\draw[dotted, gray!60] (\i,-2)--(\i,1);

\foreach \i in {0,1,2,3}
\draw[dotted, gray!60] (0,-2+\i)--(5,-2+\i);

\draw[->, black!70] (0,0)--(6,0);
\draw[->, black!70] (0,-3)--(0,2);
\draw[dotted, gray!60] (0,0)--(5,-2);

\node[below, black!70] at (6,0) {$x$};
\node[left,black!70] at (0,2) {$y$};

\draw[thick] (0,0)--(1,1)--(2,0)--(3,-1)--(4,-1)--(5,-2);
\end{tikzpicture}
\begin{tikzpicture}[scale=.34]
\foreach \i in {0,1,2,3,4,5}
\draw[dotted, gray!60] (\i,-2)--(\i,1);

\foreach \i in {0,1,2,3}
\draw[dotted, gray!60] (0,-2+\i)--(5,-2+\i);

\draw[->, black!70] (0,0)--(6,0);
\draw[->, black!70] (0,-3)--(0,2);
\draw[dotted, gray!60] (0,0)--(5,-2);

\node[below, black!70] at (6,0) {$x$};
\node[left,black!70] at (0,2) {$y$};

\draw[thick] (0,0)--(1,1)--(2,1)--(3,0)--(4,-1)--(5,-2);
\end{tikzpicture}
\begin{tikzpicture}[scale=.34]
\foreach \i in {0,1,2,3,4,5}
\draw[dotted, gray!60] (\i,-2)--(\i,1);

\foreach \i in {0,1,2,3}
\draw[dotted, gray!60] (0,-2+\i)--(5,-2+\i);

\draw[->, black!70] (0,0)--(6,0);
\draw[->, black!70] (0,-3)--(0,2);
\draw[dotted, gray!60] (0,0)--(5,-2);

\node[below, black!70] at (6,0) {$x$};
\node[left,black!70] at (0,2) {$y$};

\draw[thick] (0,0)--(1,0)--(2,1)--(3,0)--(4,-1)--(5,-2);
\end{tikzpicture}
\caption{All rational Motzkin paths of type $(5,-2)$}\label{fig:Motzkin}
\end{figure}
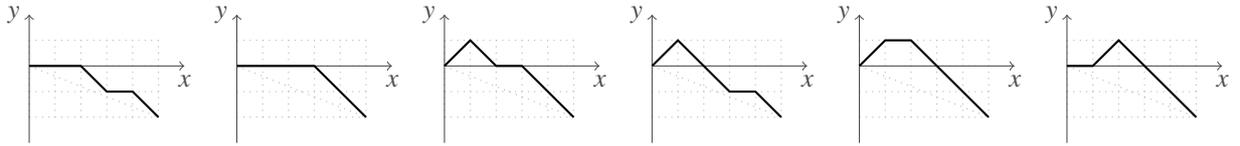

We note that if $P=P_1P_2\cdots P_{s+1}$ is a rational Motzkin path of type $(s+1,-1)$, then the last step $P_{s+1}$ must be a down step and the subpath $\bar{P}=P_1P_2\cdots P_{s}$ is a Motzkin path of length $s$.

\begin{prop}\label{prop:coremotz}
Let $s$ and $d$ be relatively prime positive integers. For $p\geq2$, there is a bijection between the set of $(s,s+d,\dots,s+pd)$-core partitions and that of rational Motzkin paths of type $(s+d,-d)$ without $UF^iU$ steps for $i=0,1,\dots,p-3$ if $p\geq 3$. 
\end{prop}

\begin{proof}
For an $(s,s+d,\dots,s+pd)$-core partition $\la$, we consider the $(s+d,d)$-abacus of $\la$. By Lemma \ref{lem:af}, we can construct a rational Motzkin path of type $(s+d,-d)$ without $UF^iU$ steps for $i=0,1,\dots,p-3$ if $p\geq 3$,
by connecting the spacers which is labeled by the smallest nonnegative number in each column.
\end{proof}

\begin{example}
If $\la=(9,5,3,2,2,1,1,1,1)$, then $\la$ is a $(5,8,11,14)$-core partition and $\beta(\la)=\{17,12,9,7,6,4,3,2,1\}$. Figure \ref{fig:motzkin2} shows the corresponding path $P$ of $\la$. The path $P=UFUDDDDD$ is a rational Motzkin path of type $(8,-3)$ without $UU$ steps.
\end{example}

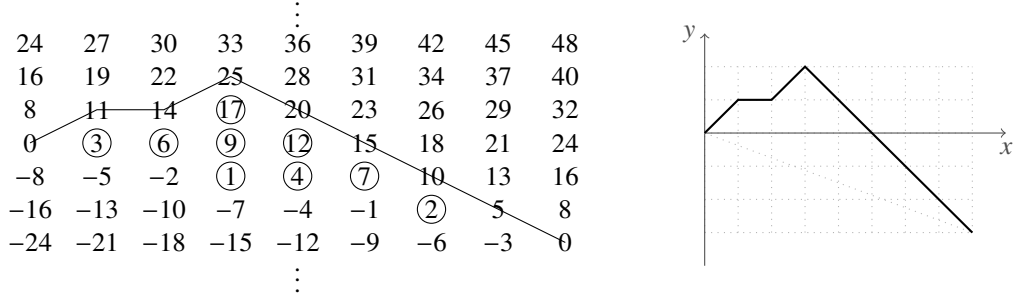
\begin{figure}[ht!]
\centering
\begin{tikzpicture}[scale=.44]
\foreach \i in {0,3,6,9,12,15,18,21,24}
\node at (\i/1.5,0) {$\i$};
\foreach \i in {8,11,14,17,20,23,26,29,32}
\node at (\i/1.5-8/1.5,1) {$\i$};
\foreach \i in {16,19,22,25,28,31,34,37,40}
\node at (\i/1.5-16/1.5,2) {$\i$};
\foreach \i in {24,27,30,33,36,39,42,45,48}
\node at (\i/1.5-24/1.5,3) {$\i$};
\foreach \i in {-8,-5,-2,1,4,7,10,13,16}
\node at (\i/1.5+8/1.5,-1) {$\i$};
\foreach \i in {-16,-13,-10,-7,-4,-1,2,5,8}
\node at (\i/1.5+16/1.5,-2) {$\i$};
\foreach \i in {-24,-21,-18,-15,-12,-9,-6,-3,0}
\node at (\i/1.5+24/1.5,-3) {$\i$};
\node at (8,4.1) {\vdots};
\node at (8,-3.9) {\vdots};
\draw (6,1) circle (12pt);
\draw (2,0) circle (12pt);
\draw (4,0) circle (12pt);
\draw (6,0) circle (12pt);
\draw (8,0) circle (12pt);
\draw (6,-1) circle (12pt);
\draw (8,-1) circle (12pt);
\draw (10,-1) circle (12pt);
\draw (12,-2) circle (12pt);
\draw (0,0)--(2,1)--(4,1)--(6,2)--(8,1)--(10,0)--(12,-1)--(14,-2)--(16,-3);
\end{tikzpicture}
\qquad \quad
\begin{tikzpicture}[scale=.44]
\node at (0,-4.8) {};
\foreach \i in {0,1,2,3,4,5,6,7,8}
\draw[dotted, gray!60] (\i,-3)--(\i,2);

\foreach \i in {0,1,2,3,4,5}
\draw[dotted, gray!60] (0,-3+\i)--(8,-3+\i);

\draw[->, black!70] (0,0)--(9,0);
\draw[->, black!70] (0,-4)--(0,3);
\draw[dotted, gray!60] (0,0)--(8,-3);

\node[below, black!70] at (9,0) {$x$};
\node[left, black!70] at (0,3) {$y$};

\draw[thick] (0,0)--(1,1)--(2,1)--(3,2)--(4,1)--(5,0)--(6,-1)--(7,-2)--(8,-3);
\end{tikzpicture}
\caption{The corresponding rational Motzkin path of the partition $(9,5,3,2,2,1,1,1,1)$}\label{fig:motzkin2}
\end{figure}

To count the number of rational Motzkin paths of type $(s+d,-d)$, we use the cyclic shifting of paths (see \cite{Bizley,Loehr}). 
For a path $P=P_1P_2\cdots P_s$, the cyclic shift $\sigma(P)$ of $P$ is $\sigma(P)=P_2P_3\cdots P_sP_1$. Iteratively, $\sigma^i(P)=P_{i+1}\cdots P_{s}P_{1}\cdots P_i$, for $i=1,\dots,s-1$ and $\sigma^0(P)=P$. 

\begin{lem}\label{lem:cyclic}
Let $s$ and $d$ be relatively prime positive integers and $P=P_1P_2\cdots P_{s+d}$ be a free rational Motzkin path of type $(s+d,-d)$. Then there exists a unique cyclic shift $\sigma^j(P)$ of $P$ such that $\sigma^j(P)$ is a rational Motzkin path of type $(s+d,-d)$ for $j=0,1,2,\dots,s+d-1$. 
\end{lem}

\begin{proof}
We define the label vector $L(P)=(\ell_0,\dots,\ell_{s+d})$ of $P$ as follows. We set $\ell_0=0$, $\ell_i=\ell_{i-1}+s+2d$ if $P_i=U$, $\ell_i=\ell_{i-1}+d$ if $P_i=F$, and $\ell_i=\ell_{i-1}-s$ if $P_i=D$. According to $s$ and $d$ are relatively prime, $\ell_0,\ell_1,\dots,\ell_{s+d}$ are distinct, except $\ell_0=\ell_{s+d}=0$. 
We note that for a free rational Motzkin path $P$, $P$ is a rational Motzkin path if and only if every elements of $L(P)$ are nonnegative. 
By the definition of $L$, $L(\sigma^i(P))=(0,\ell_{i+1}-\ell_{i},\ell_{i+2}-\ell_{i},\dots,\ell_{s+d}-\ell_{i},\ell_{1}-\ell_{i},\ell_{2}-\ell_{i},\dots,\ell_{i}-\ell_{i})$. Therefore if $\ell_j$ is the smallest element of $L(P)$, then elements of $L(\sigma^j(P))$ are all nonnegative. Also, the smallest element of $L(\sigma^i(P))$ is $\ell_{j}-\ell_{i}$, so $\sigma^i(P)$'s are all distinct. Moreover, if $j\neq i$, then $\sigma^j(P)$ is not a rational Motzkin path.
\end{proof}

\begin{example}
Let $P=DDFUD$ be a free rational Motzkin path of type $(5,-2)$. Then, $L(P)=(0,-3,-6,-4,3,0)$. The smallest element of $L(P)$ is $\ell_2=-6$, so that $L(\sigma^2(P))=(0,2,9,6,3,0)$ and $\sigma^2(P)=FUDDD$ is a rational Motzkin path of type $(5,-2)$ as in Figure \ref{fig:cyclic}.
\end{example}

\begin{figure}[ht!]
\centering
\begin{tikzpicture}[scale=.34]
\foreach \i in {0,1,2,3,4,5}
\draw[dotted, gray!60] (\i,-3)--(\i,1);

\foreach \i in {0,1,2,3,4}
\draw[dotted, gray!60] (0,-3+\i)--(5,-3+\i);

\draw[->, black!70] (0,0)--(6,0);
\draw[->, black!70] (0,-4)--(0,2);
\draw[dotted, gray!60] (0,0)--(5,-2);

\node[below, black!70] at (6,0) {$x$};
\node[left, black!70] at (0,2) {$y$};

\draw[thick] (0,0)--(1,-1)--(2,-2)--(3,-2)--(4,-1)--(5,-2);
\end{tikzpicture}
\quad\begin{tikzpicture}[scale=.34]
\foreach \i in {0,1,2,3,4,5}
\draw[dotted, gray!60] (\i,-3)--(\i,1);

\foreach \i in {0,1,2,3,4}
\draw[dotted, gray!60] (0,-3+\i)--(5,-3+\i);

\draw[->, black!70] (0,0)--(6,0);
\draw[->, black!70] (0,-4)--(0,2);
\draw[dotted, gray!60] (0,0)--(5,-2);

\node[below, black!70] at (6,0) {$x$};
\node[left, black!70] at (0,2) {$y$};

\draw[thick] (0,0)--(1,-1)--(2,-1)--(3,0)--(4,-1)--(5,-2);
\end{tikzpicture}
\quad\begin{tikzpicture}[scale=.34]
\foreach \i in {0,1,2,3,4,5}
\draw[dotted, gray!60] (\i,-3)--(\i,1);

\foreach \i in {0,1,2,3,4}
\draw[dotted, gray!60] (0,-3+\i)--(5,-3+\i);

\draw[->, black!70] (0,0)--(6,0);
\draw[->, black!70] (0,-4)--(0,2);
\draw[dotted, gray!60] (0,0)--(5,-2);

\node[below, black!70] at (6,0) {$x$};
\node[left, black!70] at (0,2) {$y$};

\draw[thick] (0,0)--(1,0)--(2,1)--(3,0)--(4,-1)--(5,-2);
\end{tikzpicture}
\quad\begin{tikzpicture}[scale=.34]
\foreach \i in {0,1,2,3,4,5}
\draw[dotted, gray!60] (\i,-3)--(\i,1);

\foreach \i in {0,1,2,3,4}
\draw[dotted, gray!60] (0,-3+\i)--(5,-3+\i);

\draw[->, black!70] (0,0)--(6,0);
\draw[->, black!70] (0,-4)--(0,2);
\draw[dotted, gray!60] (0,0)--(5,-2);

\node[below, black!70] at (6,0) {$x$};
\node[left, black!70] at (0,2) {$y$};

\draw[thick] (0,0)--(1,1)--(2,0)--(3,-1)--(4,-2)--(5,-2);
\end{tikzpicture}
\quad\begin{tikzpicture}[scale=.34]
\foreach \i in {0,1,2,3,4,5}
\draw[dotted, gray!60] (\i,-3)--(\i,1);

\foreach \i in {0,1,2,3,4}
\draw[dotted, gray!60] (0,-3+\i)--(5,-3+\i);

\draw[->, black!70] (0,0)--(6,0);
\draw[->, black!70] (0,-4)--(0,2);
\draw[dotted, gray!60] (0,0)--(5,-2);

\node[below, black!70] at (6,0) {$x$};
\node[left, black!70] at (0,2) {$y$};

\draw[thick] (0,0)--(1,-1)--(2,-2)--(3,-3)--(4,-3)--(5,-2);
\end{tikzpicture}
\caption{Five cyclic shifts of the path $P=DDFUD$}\label{fig:cyclic}
\end{figure}
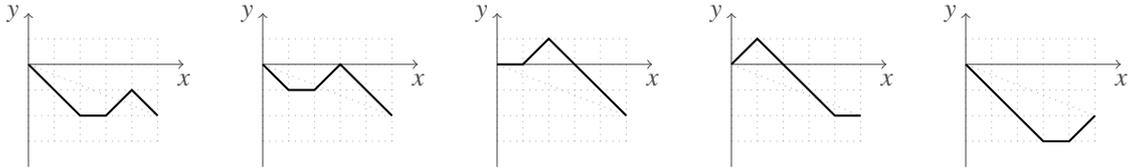

Now we can enumerate the rational Motzkin paths of type $(s+d,-d)$.

\begin{prop}\label{prop:freemotz}
Let $s$ and $d$ be relatively prime positive integers. For a given integer $0\leq k\leq \lfloor s/2\rfloor$, the number of rational Motzkin paths of type $(s+d,-d)$ having $k$ up steps is 
\[
\frac{1}{s+d}\,\binom{s+d}{k,k+d,s-2k}\,.
\]
Consequently, the number of rational Motzkin paths of type $(s+d,-d)$ is 
\[
\frac{1}{s+d}\sum_{k=0}^{\lfloor s/2 \rfloor}\binom{s+d}{k,k+d,s-2k}\,.
\]
\end{prop}

\begin{proof}
It is immediate that the number of free rational Motzkin paths of type $(s+d,-d)$ having $k$ up steps is $\binom{s+d}{k,k+d,s-2k}$. By Lemma \ref{lem:cyclic}, $(s+d)$ times the number of rational Motzkin paths of type $(s+d,-d)$ having $k$ up steps is equal to the number of free rational Motzkin paths of type $(s+d,-d)$ having $k$ up steps. This completes the proof.
\end{proof}

By Propositions \ref{prop:coremotz} and \ref{prop:freemotz}, we give an alternating proof of Theorem \ref{thm:Wang} using path enumeration. Also, we can use the cyclic shifting for rational Motzkin paths of type $(s+d,-d)$ without $UF^iU$ steps for $i=0,1,\dots,p-3$. We say that a path $P$ is a free rational Motzkin paths of type $(s+d,-d)$ without \emph{cyclic} $UF^iU$ steps if there is no $UF^iU$ steps for any cyclic shift of $P$.

\begin{prop}\label{prop:mainprop}
Let $s$ and $d$ be relatively prime positive integers. For integers $p\geq 3$ and $1\leq k\leq \lfloor s/2 \rfloor$, the number of rational Motzkin paths of type $(s+d,-d)$ having $k$ up steps and no $UF^iU$ steps for all $i=0,1,\dots,p-3$ is 
\[
\frac{1}{k+d}\sum_{\ell=0}^{r}\binom{k+d}{k-\ell}\binom{k-1}{\ell}\binom{s+d-\ell(p-2)-1}{2k+d-1}\,,
\]
where $r=\min(k-1,\,\lfloor (s-2k)/(p-2) \rfloor)$.
\end{prop}

\begin{proof}
Let $\mathcal{M}_D(s+d,-d;k,p)$ be the set of free rational Motzkin paths of type $(s+d,-d)$ consisting of $k$ $U$'s, $k+d$ $D$'s, and $s-2k$ $F$'s which starts with a down step and has no cyclic $UF^iU$ steps for all $i=0,1,\dots,p-3$.  
From Lemma \ref{lem:cyclic}, there are $k+d$ cyclic shifts of a rational Motzkin path, which starts with a down step so that the number of rational Motzkin paths of type $(s+d,-d)$ with $k$ up steps and without $UF^iU$ steps for all $i=0,1,\dots,p-3$ is  
\[
\frac{1}{k+d}\,|\mathcal{M}_D(s+d,-d;k,p)|\,.
\]

For a path $P$ belonging to $\mathcal{M}_D(s+d,-d;k,p)$, let $\tilde{P}$ denote the subpath obtained from $P$ by deleting all flat steps. Then, $\tilde{P}=Q_1Q_2\cdots Q_{2k+d}$ is a path consisting of $k$ $U$'s and $k+d$ $D$'s which starts with a down step. Now, we partition $\mathcal{M}_D(s+d,-d;k,p)$ into $k$ sets according to the number of $UU$ steps of $\tilde{P}$. For $0\leq \ell \leq k-1$, let $\mathcal{M}_D^{\ell}(s+d,-d;k,p)$ be the set of $P\in\mathcal{M}_D(s+d,-d;k,p)$ for which $\tilde{P}$ has $\ell$ $UU$ steps so that 
\[
|\mathcal{M}_D(s+d,-d;k,p)|=\sum_{\ell=0}^{k-1}|\mathcal{M}_D^{\ell}(s+d,-d;k,p)|\,.
\]
Hence, it is enough to show that 
\[
|\mathcal{M}_D^{\ell}(s+d,-d;k,p)|=\binom{k+d}{k-\ell}\binom{k-1}{\ell}\binom{s+d-\ell(p-2)-1}{2k+d-1}\,.
\]
We note that if a path $P$ belongs to $\mathcal{M}_D^{\ell}(s+d,-d;k,p)$, then $\tilde{P}=Q_1Q_2\cdots Q_{2k+d}$ is a path of the form 
\[
D^{a_1}U^{b_1}D^{a_2}U^{b_2}\cdots D^{a_{k-\ell}}U^{b_{k-\ell}}D^{a_{k-\ell+1}},
\]
where $a_i$ and $b_i$ are integers satisfying $a_i,b_i\geq 1$ for $i=1,2,\dots,k-\ell$ and $a_{k-\ell+1}\geq0$,
\[
a_1+a_2+\cdots +a_{k-\ell+1}=k+d \qquad \text{and} \qquad b_1+b_2+\cdots+b_{k-\ell}=k\,.
\] 
Since $P$ can be written as
\[
Q_1F^{c_1}Q_2F^{c_2}\cdots Q_{2k+d}F^{c_{2k+d}},  
\]      
where $c_i$'s are nonnegative integers satisfying
$c_1+c_2+\cdots +c_{2k+d}=s-2k$ and $c_i\geq p-2$ if $Q_i=Q_{i+1}=U$, one can see that  
$|\mathcal{M}_D^{\ell}(s+d,-d;k,p)|$ is equal to the number of solution tuples $((a_i),(b_i),(c_i))$.
It is easy to see that the number of solutions $(a_i)$ and $(b_i)$ are  
$\binom{k+d}{k-\ell}$ and $\binom{k-1}{k-\ell-1}=\binom{k-1}{\ell}$, respectively. If $(a_i)$ and $(b_i)$ are given, then they determine $\ell$ indices $i$ such that $c_i\geq p-2$. Hence, the number of solutions $(c_i)$ is equal to the number of nonnegative integer solutions to 
$y_1+y_2+\cdots+y_{2k+d}=s-2k-\ell(p-2)$, that is $\binom{s+d-\ell(p-2)-1}{2k+d-1}$, for $\ell$ satisfying that $s+d-\ell(p-2)-1\geq 2k+d-1$. This completes the proof.
\end{proof}

\begin{rem}
The number of rational Motzkin paths of type $(s+d,-d)$ without up step and $UF^iU$ steps for all $i=0,1,\dots,p-3$ is equal to the number of rational Motzkin paths of type $(s+d,-d)$ without up step, that is $\binom{s+d}{d}/(s+d)$ by Proposition \ref{prop:freemotz}. It follows from Propositions \ref{prop:coremotz} and \ref{prop:mainprop} that we have proven Theorem \ref{thm:main}.
\end{rem}


\section{The $(s,s+1,\dots,s+p)$-core partitions revisited}\label{sec:3}

From Theorem \ref{thm:main}, we obtain a closed formula for the number of $(s,s+1,\dots,s+p)$-core partitions. 

\begin{cor}\label{cor:forone} 
For positive integers $s$ and $p\geq2$, the number of $(s,s+1,\dots,s+p)$-core partitions is equal to the number of Motzkin paths of length $s$ without $UF^iU$ steps for $i=0,1,\dots,p-3$ if $p\geq 3$, that is  
\[
1+\sum_{k=1}^{\lfloor s/2 \rfloor}\sum_{\ell=0}^{r}N(k,\ell+1)\binom{s-\ell(p-2)}{2k},
\]
where $N(k,\ell+1)=\frac{1}{k}\binom{k}{\ell+1}\binom{k}{\ell}=\frac{1}{k+1}\binom{k+1}{\ell+1}\binom{k-1}{\ell}$ is the Narayana number which counts the number of Dyck paths of order 
$k$ having $\ell+1$ peaks and $r=\min(k-1,\,\lfloor (s-2k)/(p-2) \rfloor)$.
\end{cor}

Combining Theorem \ref{thm:gendyck} and Corollary \ref{cor:forone}, we see that the $(s,p)$-generalized Dyck paths and the Motzkin paths of length $s$ without $UF^iU$ steps for $i=0,1,\dots,p-3$ if $p\geq 3$
are equinumerous. We now provide a bijection between sets of these paths.  

\subsection{A bijection between the generalized Dyck paths and the rational Motzkin paths}\label{sec:bij}

We first note that for a given $p\geq 2$, if $P$ is a Motzkin path of length $s$ without $UF^iU$ steps for $i=0,1,\dots,p-3$ if $p\geq 3$, then each $U$ step of $P$ is followed by either $F^jD$ for some $j\geq 0$ or $F^{k}U$ for some $k\geq p-2$. Hence, we can decompose $P$ into the following $p+1$ units:
\begin{eqnarray*}
&& \bar{U}_p:=UF^{p-2}\\
&& \bar{D}_p:=D \,\qquad\quad \text{(which is not following $UF^{i}$ for all $i=0,1,\dots,p-3$)}\\
&& \bar{F}_1:=F \qquad ~\,\quad\text{(which is not following $UF^{i}$ for all $i=0,1,\dots,p-3$)}\\
&& \bar{F}_{i}:=UF^{i-2}D \quad \text{for} ~i=2,3,\dots,p-1\,.
\end{eqnarray*}

We now construct a simple bijection $\phi$ between the set of $(s,p)$-generalized Dyck paths and that of Motzkin paths of length $s$ without $UF^iU$ steps for $i=0,1,\dots,p-3$ if $p\geq 3$, for fixed $p\geq2$ as follows.

For a given Motzkin path $P$ of length $s$ without $UF^iU$ steps for $i=0,1,\dots,p-3$ if $p\geq 3$, we define $\phi(P)$ to be the path obtained from $P$ by replacing each unit $\bar{A}$ with $A$ for $A\in\{U_p,D_p,F_i\,|\,i=1,2,\dots,p-1\}$, where $U_p=(0,p)$, $D_p=(p,0)$, $F_i=(i,i)$ for $i=1,2,\dots,p-1$.

We note that if $P$ is decomposed into $k$ $\bar{U}_p$'s, $k$ $\bar{D}_p$'s, and $c_i$ $\bar{F}_i$'s for $i=1,2,\dots,p-1$, then $\phi(P)$ is a path from $(0,0)$ to $(s,s)$ since
\[
k(p-1)+k+\sum_{i=1}^{p-1}i c_i =s=kp+\sum_{i=1}^{p-1}i c_i\,.
\] 
Moreover, $P$ never goes below the $x$-axis if and only if $\phi(P)$ never goes below the line $y=x$. Hence, $\phi(P)$ is an $(s,p)$-generalized Dyck path, and therefore $\phi$ is a bijection between the set of $(s,p)$-generalized Dyck paths and that of Motzkin paths of length $s$ without $UF^iU$ steps for $i=0,1,\dots,p-3$ if $p\geq 3$.

\begin{example} Let $p=4$ and $P=UFFUFFFUDDUFDUFFDD$ so that $P$ is a Motzkin path of length $18$ 
without $UU$ and $UFU$ steps. Hence, $P$ can be written as  $P=\bar{U}_4\bar{U}_4\bar{F}_1\bar{F}_2\bar{D}_4\bar{F}_3\bar{U}_4\bar{D}_4\bar{D}_4$, and therefore $Q=\phi(P)=U_4 U_4 F_1 F_2 D_4 F_3 U_4 D_4 D_4$ which is an $(18,4)$-generalized Dyck path. See Figure \ref{fig:bijection}.

\begin{figure}[ht!]
\begin{center}
\begin{tikzpicture}[scale=.4]
\node[above] at (9,3) {$P$};
\foreach \i in {0,1,2,...,18}
\draw[dotted, gray!60] (\i,0)--(\i,3);

\foreach \i in {1,2,3}
\draw[dotted, gray!60] (0,\i)--(18,\i);

\draw[->, black!70] (0,0)--(19,0);
\draw[->, black!70] (0,0)--(0,4);

\node[below,black!70] at (19,0) {$x$};
\node[left,black!70] at (0,4) {$y$};

\draw[thick] (0,0)--(1,1)--(3,1)--(4,2)--(7,2)--(8,3)--(9,2)--(10,1)
--(11,2)--(12,2)--(13,1)--(14,2)--(16,2)--(17,1)--(18,0);
 
\filldraw[white] (3,1) circle (2pt)
 (6,2) circle (2pt)
 (7,2) circle (2pt)
 (9,2) circle (2pt)
 (10,1) circle (2pt)
 (13,1) circle (2pt)  
 (16,2) circle (2pt)
 (17,1) circle (2pt);
\end{tikzpicture}

\begin{tikzpicture}[scale=.4]
\node[above] at (9,18) {$Q$};
\foreach \i in {0,1,2,...,18}
\draw[dotted, gray!60] (\i,0)--(\i,18);

\foreach \i in {1,2,...,18}
\draw[dotted, gray!60] (0,\i)--(18,\i);

\draw[->, black!70] (0,0)--(19,0);
\draw[->, black!70] (0,0)--(0,19);
\draw[dotted, gray!60] (0,0)--(18,18);

\node[below,black!70] at (19,0) {$x$};
\node[left,black!70] at (0,19) {$y$};

\draw[thick] (0,0)--(0,8)--(3,11)--(7,11)--(10,14)--(10,18)--(18,18);
 
\filldraw[white] (0,4) circle (2pt)
 (0,8) circle (2pt)
 (1,9) circle (2pt)
 (3,11) circle (2pt)
 (7,11) circle (2pt)
 (10,14) circle (2pt)  
 (10,18) circle (2pt)
 (14,18) circle (2pt);
\end{tikzpicture}
\end{center}
\caption{A Motzkin path and the corresponding generalized Dyck path}\label{fig:bijection}
\end{figure}

\end{example}

\subsection{The $(s,s+1,\dots,s+p)$-partitions with $k$ corners}
For a partition $\la$, the number of distinct parts in $\la$ is equal to the number of corners in the Young diagram of $\la$. Many researchers were interested in corners of a partition, and Huang-Wang \cite{HW} found formulae for some simultaneous cores with specified number of corners.

\begin{thm}\cite[Theorems 3.1 and 3.8]{HW}
For positive integers $s$ and $k$, the number of $(s,s+1)$-core partitions with $k$ corners is the Narayana number $N(s,k+1)=\frac{1}{s}\binom{s}{k+1}\binom{s}{k}$, 
and the number of $(s,s+1,s+2)$-core partitions with $k$ corners is $\binom{s}{2k}C_k$, where $C_k$ is the $k$th Catalan number.
\end{thm}

Huang-Wang also suggested an open problem for enumerating $(s,s+1,\dots,s+p)$-cores with $k$ corners, and we give an answer for this problem.

\begin{thm}
For positive integers $s$, $p\geq2$, and $1\leq k\leq \lfloor s/2 \rfloor$, the number of $(s,s+1,\dots,s+p)$-core partitions with $k$ corners is \[
\sum_{\ell=0}^{r}N(k,\ell+1)\binom{s-\ell(p-2)}{2k}\,,
\]
where $r=\min(k-1,\,\lfloor (s-2k)/(p-2) \rfloor)$.
\end{thm}

\begin{proof}
For an $(s,s+1,\dots,s+p)$-core partition $\la$, we consider the $(s+1)$-abacus of $\la$. We note that if $\la_i=\la_{i+1}$, then $h(i,1)=h(i+1,1)+1$ so that we have consecutive beads in its $(s+1)$-abacus. Therefore, the number of corners in $\la$ is matching with the number of runs of consecutive beads in its $(s+1)$-abacus, being equal to the number of spacers that is followed by a bead. In the corresponding Motzkin path, each up step appears when we have a spacer that is followed by a bead. Hence, the number of corners in $\la$ is equal to the number of up steps in the corresponding Motzkin path. By letting $d=1$ in Proposition \ref{prop:mainprop} and Corollary \ref{cor:forone}, this completes the proof. 
\end{proof}

\begin{example}
Let $\la=(9,2,2,2,1)$ so that $\la$ is a $(7,8,9)$-core partition and $\beta(\la)=\{13,6,5,4,1\}$. Figure \ref{fig:corner} shows the $8$-abacus of $\la$ and its corresponding Motzkin path.  We can easily check that the number of corners in $\la$ and the number of up steps in its corresponding Motzkin path is three.
\end{example}

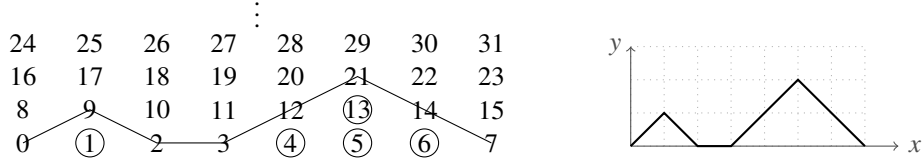
\begin{figure}[ht!]
\centering
\begin{tikzpicture}[scale=.44]
\foreach \i in {0,1,2,3,4,5,6,7}
\node at (2*\i,0) {$\i$};
\foreach \i in {8,9,10,11,12,13,14,15}
\node at (2*\i-16,1) {$\i$};
\foreach \i in {16,17,18,19,20,21,22,23}
\node at (2*\i-32,2) {$\i$};
\foreach \i in {24,25,26,27,28,29,30,31}
\node at (2*\i-48,3) {$\i$};
\node at (7,4.1) {\vdots};
\draw (2,0) circle (12pt);
\draw (8,0) circle (12pt);
\draw (10,0) circle (12pt);
\draw (12,0) circle (12pt);
\draw (10,1) circle (12pt);
\draw (0,0)--(2,1)--(4,0)--(6,0)--(8,1)--(10,2)--(12,1)--(14,0);
\end{tikzpicture}
\qquad \quad
\begin{tikzpicture}[scale=.44]
\foreach \i in {0,1,2,3,4,5,6,7}
\draw[dotted, gray!60] (\i,0)--(\i,3);

\foreach \i in {0,1,2,3}
\draw[dotted, gray!60] (0,\i)--(7,\i);

\draw[->, black!70] (0,0)--(8,0);
\draw[->, black!70] (0,0)--(0,3);

\node[right, black!70] at (8,0) {$x$};
\node[left, black!70] at (0,3) {$y$};

\draw[thick] (0,0)--(1,1)--(2,0)--(3,0)--(4,1)--(5,2)--(6,1)--(7,0);
\end{tikzpicture}
\caption{The $8$-abacus of the $(7,8,9)$-core partition $(9,2,2,2,1)$ and the corresponding Motzkin path}\label{fig:corner}.
\end{figure}

\subsection{Self-conjugate $(s,s+1,\dots,s+p)$-core partitions}

A partition whose conjugate is equal to itself is called \emph{self-conjugate}. From now on, we focus on self-conjugate partitions. Ford-Mai-Sze \cite{FMS} found the number of self-conjugate $(s,t)$-core partitions.

\begin{thm}\cite[Theorem 1]{FMS} \label{thm:FMS}
For relatively prime integers $s$ and $t$, the number of self-conjugate $(s,t)$-core partitions is 
\[
\binom{\lfloor \frac{s}{2} \rfloor + \lfloor \frac{t}{2} \rfloor}{\lfloor \frac{s}{2} \rfloor}\,.
\]
In particular, the number of self-conjugate $(s,s+1)$-core partitions is equal to the number of symmetric Dyck paths of order $s$, that is $\binom{s}{\lfloor s/2 \rfloor}$.
\end{thm}

Motivated by Theorem \ref{thm:FMS}, in the previous work \cite{CHS}, authors showed that the number of self-conjugate $(s,s+1,s+2)$-core partitions is equal to the number of symmetric Motzkin paths of length $s$, and then gave a conjecture for the number of self-conjugate $(s,s+1,\dots,s+p)$-cores. Recently, this was proved by Yan-Yu-Zhou.

\begin{thm}\cite[Theorems 2.14, 2.19, and 2.22]{YYZ} \label{prop:YYZ}
For positive integers $s$ and $p$, the number of self-conjugate $(s,s+1,\dots,s+p)$-core partitions is equal to the number of symmetric $(s,p)$-generalized Dyck paths.
\end{thm}

In this subsection, we give a closed formula for the number of self-conjugate $(s,s+1,\dots,s+p)$-core partitions. Here, we give a useful lemma from the OEIS.

\begin{lem}\cite[Sequence A088855]{OEIS} \label{lem:symdyck}
For nonnegative integers $k$ and $\ell$ such that $\ell<k$, the number of symmetric Dyck paths of order $k$ having $\ell$ $UU$ steps is 
\[
\binom{\lfloor \frac{k-1}{2} \rfloor}{\lfloor \frac{\ell}{2} \rfloor} \binom{\lfloor \frac{k}{2} \rfloor}{\lfloor \frac{\ell+1}{2} \rfloor}\,.
\]
\end{lem}

\begin{thm} \label{thm:selfone}
For positive integers $s$ and $p\geq 2$, the number of self-conjugate $(s,s+1,\dots,s+p)$-core partitions is
\[
1+\sum_{k=1}^{\lfloor s/2 \rfloor}\sum_{\ell=0}^{r}\binom{\lfloor \frac{k-1}{2} \rfloor}{\lfloor \frac{\ell}{2} \rfloor} \binom{\lfloor \frac{k}{2} \rfloor}{\lfloor \frac{\ell+1}{2} \rfloor}\binom{\lfloor \frac{s-\ell(p-2)}{2} \rfloor}{k}\,,
\]
where $r=\min(k-1,\,\lfloor (s-2k)/(p-2) \rfloor)$.
\end{thm}

\begin{proof}
By Theorem \ref{prop:YYZ}, the number $sc_{(s,s+1,\dots,s+p)}$ of self-conjugate $(s,s+1,\dots,s+p)$-core partitions is equal to the number of symmetric $(s,p)$-generalized Dyck paths. 

Let $\mathcal{S}(s,p;k)$ be the set of symmetric $(s,p)$-generalized Dyck paths of length $s$ having $k$ $U_p=(0,p),F_2,F_3,\dots,F_{p-1}$ steps in total so that 
\[
sc_{(s,s+1,\dots,s+p)}=\sum_{k=0}^{\lfloor s/2 \rfloor} |\mathcal{S}(s,p;k)|\,.
\]
For instance, $|\mathcal{S}(s,p;0)|=1$ since if $Q$ is a symmetric $(s,p)$-generalized Dyck path without $U_p,F_2, \dots, F_{p-1}$ steps, then $Q=F_1^{s}$. We note that if $Q\in\mathcal{S}(s,p;k)$, then $\phi^{-1}(Q)$ has $k$ up steps. However, $\phi^{-1}(Q)$ may not be symmetric even though $Q$ is symmetric. Therefore, to simplify calculation, we symmetrize $\phi^{-1}(Q)$ by deleting some flat steps. 

Now, we consider $k\geq1$ and $0\leq \ell \leq k-1$. Let $\mathcal{S}(s,p;k,\ell)$ be the set of $Q\in \mathcal{S}(s,p;k)$ having $\ell$ $U_p$ steps which is not followed by $F_1^{i}D_p$ for any $i\geq 0$. We note that for $Q\in \mathcal{S}(s,p;k,\ell)$, if we replace each $UF^{p-2}$ with $U$, where these $UF^{p-2}$ steps are corresponding to $U_p$ that is not followed by $F_1^{i}D_p$ for any $i\geq 0$, then the resulting path is a symmetric Motzkin path of length $s-\ell(p-2)$ having $\ell$ up steps which is followed by $F^jU$ for some $j\geq0$. Conversely, for a symmetric Motzkin path of length $s-\ell(p-2)$ having $\ell$ up steps which is followed by $F^jU$ for some $j\geq0$, if we replace each of these $\ell$ $U$'s with $UF^{p-2}$, and 
apply $\phi$ to the resulting path, then we obtain a path belonging to $\mathcal{S}(s,p;k,\ell)$. It leads us to that $|\mathcal{S}(s,p;k,\ell)|$ is equal to the number of symmetric Motzkin paths of length $s-\ell(p-2)$ having $k$ up steps such that there are $\ell$ up steps followed by $F^{j}U$ steps for $j\geq0$. 

We note that a symmetric Motzkin path of length $s-\ell(p-2)$ having $k$ up steps is of the form 
\[
F^{x_0}R_1F^{x_1}R_2\cdots R_kF^{x_k}F^{\epsilon}F^{x_k}R_k \cdots R_2 F^{x_1} R_1 F^{x_0},
\]
where $R=R_1R_2\cdots R_k R_k \cdots R_2 R_1$ is a symmetric Dyck path of order $k$ with $\ell$ $UU$ steps, $\epsilon\in\{0,1\}$, and
\[
x_0+x_1+\cdots+x_k=\left\lfloor \frac{s-\ell(p-2)-2k}{2} \right\rfloor\,.
\]
It follows from Lemma \ref{lem:symdyck}, there are $\binom{\lfloor (k-1)/2\rfloor}{\lfloor \ell/2 \rfloor}\binom{\lfloor k/2\rfloor}{\lfloor (\ell+1)/2 \rfloor}$ symmetric Dyck paths of order $k$ with $\ell$ $UU$ steps. Since the number of nonnegative integer solutions to $x_0+x_1+\cdots+x_{k}=\lfloor \{s-\ell(p-2)-2k\}/2 \rfloor$ 
is $\binom{\lfloor \frac{s-\ell(p-2)}{2}\rfloor}{k}$, we have 
\[
|\mathcal{S}(s,p;k,\ell)|= \binom{\lfloor \frac{k-1}{2}\rfloor}{\lfloor \frac{\ell}{2} \rfloor}\binom{\lfloor \frac{k}{2}\rfloor}{\lfloor \frac{\ell+1}{2} \rfloor}\binom{\lfloor \frac{s-\ell(p-2)}{2}\rfloor}{k}\,,
\]
for $\ell\geq0$ satisfying that $s-\ell(p-2)\geq0$ and $\lfloor \{s-\ell(p-2)\}/2\rfloor\geq k$. Summing over $\ell$ gives the cardinality of $S(s,p;k)$. It completes the proof.

\end{proof}


\section{Concluding remarks}
It would be very interesting if one can find a lattice path interpretation or a closed formula for self-conjugate $(s,s+d,\dots,s+pd)$-core partitions, for relatively prime positive integers $s$ and $d\geq 2$.


\section*{Acknowledgments}
The authors would like to thank Byungchan Kim for his valuable comments. The second named author’s work was supported by the National Research Foundation of Korea (NRF) grant funded by the Korea government (MSIP) (2016R1A5A1008055). The third named author’s work was supported by the National Research Foundation of Korea (NRF) NRF-2017R1A2B4009501.

\bibliographystyle{plain}  
\bibliography{mybib} 

\end{document}